\documentclass[11pt]{article}    

\usepackage[margin=0.75in]{geometry} 
\usepackage{amsmath,amssymb,amsthm}

\newtheorem{theorem}{Theorem}[section]
\newtheorem{proposition}[theorem]{Proposition}
\newtheorem{lemma}[theorem]{Lemma}

\newtheorem{remark}{Remark}

\newcommand{\al}{\alpha}
\newcommand{\bt}{\beta}
\newcommand{\la}{\lambda}

\newcommand{\s}{\sigma}

\newcommand{\be}{\begin{equation}}
\newcommand{\ee}{\end{equation}}
\newcommand{\bea}{\begin{eqnarray}}
\newcommand{\eea}{\end{eqnarray}}
\newcommand{\ba}{\begin{align}}
\newcommand{\ea}{\end{align}}
\newcommand{\no}{\nonumber}

\numberwithin{equation}{section}
\begin{document}

\title{Orthogonal Polynomials with a Singularly Perturbed Airy Weight}
\author{Chao Min\thanks{School of Mathematical Sciences, Huaqiao University, Quanzhou 362021, China; Email: chaomin@hqu.edu.cn}\: and Yuan Cheng\thanks{School of Mathematical Sciences, Huaqiao University, Quanzhou 362021, China}}


\date{March 13, 2024}
\maketitle
\begin{abstract}
We study the monic orthogonal polynomials with respect to a singularly perturbed Airy weight. By using Chen and Ismail's ladder operator approach, we derive a discrete system satisfied by the recurrence coefficients for the orthogonal polynomials. We find that the orthogonal polynomials satisfy a second-order linear ordinary differential equation, whose coefficients are all expressed in terms of the recurrence coefficients. By considering the time evolution, we obtain a system of differential-difference equations satisfied by the recurrence coefficients. Finally, we study the asymptotics of the recurrence coefficients when the degrees of the orthogonal polynomials tend to infinity.
\end{abstract}

$\mathbf{Keywords}$: Orthogonal polynomials; Singularly perturbed Airy weight; Ladder operators;

Recurrence coefficients; Differential and difference equations; Asymptotics.

$\mathbf{Mathematics\:\: Subject\:\: Classification\:\: 2020}$: 33C45, 42C05.

\section{Introduction}
As is well-known, classical orthogonal polynomials (e.g., Hermite, Laguerre and Jacobi polynomials) are orthogonal with respect to a weight function $w(x)$ on the real line which satisfies the Pearson equation
\be\label{pe}
\frac{d}{dx}(\sigma(x)w(x))=\tau(x)w(x),
\ee
where $\sigma(x)$ is a polynomial of degree $\leq 2$ and $\tau(x)$ is a polynomial of degree 1.
Semi-classical orthogonal polynomials have a weight function $w(x)$ that satisfies the Pearson equation (\ref{pe}) where $\sigma(x)$ and $\tau(x)$ are polynomials with deg $\sigma>2$ or deg $\tau\neq 1$. See, e.g., \cite[Section 1.1.1]{VanAssche}.

Various semi-classical orthogonal polynomials have been studied during the past few years. For example, very recently, Clarkson and Jordaan \cite{CJ} considered the orthogonal polynomials with respect to the so-called generalized Airy weight
$$
w(x)={x}^\lambda\mathrm{e}^{-\frac{1}{3}x^3+tx},\qquad x\in \mathbb{R}^{+}
$$
with parameters $\lambda> -1$ and $t\in\mathbb{R}$. They derived the differential and difference equations satisfied by the orthogonal polynomials and also by the recurrence coefficients, and investigated various asymptotic properties of the recurrence coefficients.
Orthogonal polynomials associated with the exponential cubic weight have also been studied in e.g. \cite{Magnus,MS,VFZ}, and have important applications in numerical analysis \cite{DHK} and random matrix theory \cite{BD,BD1,BDY}.

In the present paper, we are concerned with the monic orthogonal polynomials with respect to the singularly perturbed Airy weight
\be\label{weight}
w(x;t)={x}^\lambda\mathrm{e}^{-x^3-\frac{t}{x}},\qquad x\in \mathbb{R}^{+}
\ee
with parameters $\lambda> -1$ and $t>0$. The weight (\ref{weight}) is a semi-classical weight since it satisfies the Pearson equation (\ref{pe}) with
$$
\s(x)=x^2,\qquad\qquad \tau(x)=-3 x^4+(\lambda +2) x+t.
$$
Note that the factor $\mathrm{e}^{-\frac{t}{x}}$ induces an infinitely strong zero at the origin for the weight (\ref{weight}).

Semi-classical orthogonal polynomials with singularly perturbed Gaussian, Laguerre, Jacobi and Freud weights have been studied in \cite{BMM,ChenDai,ChenIts,MCC,MLC,MW,Xu2015}.
The weights with an essential singularity at the origin play an important role in many mathematical and physical problems, such as the study of statistics for zeros of the Riemann zeta function \cite{Berry}, the calculation of finite temperature expectation values in integrable quantum field theory \cite{Lukyanov}, the study of the Wigner time-delay distribution \cite{Brouwer,MS2013,Texier}, etc.

Let $P_{n}(x;t),\; n=0,1,2,\ldots$, be the monic polynomials of degree $n$ orthogonal with respect to the weight (\ref{weight}), such that
$$
\int_{0}^{\infty}P_{m}(x;t)P_{n}(x;t)w(x;t)dx=h_{n}(t)\delta_{mn},\qquad m, n=0,1,2,\ldots,
$$
where $h_n(t)>0$ and $\delta_{mn}$ denotes the Kronecker delta.
Here $P_{n}(x;t)$ has the following expansion
$$
P_{n}(x;t)=x^{n}+\mathrm{p}(n,t)x^{n-1}+\cdots+P_n(0;t),\qquad n=0,1,2,\ldots,
$$
where $\mathrm{p}(n,t)$ denotes the sub-leading coefficient of $P_{n}(x;t)$ with the initial value $\mathrm{p}(0,t)=0$.

One of the most important characteristics of orthogonal polynomials is the fact that they obey the three-term recurrence relation of the form \cite{Chihara,Szego}
\be\label{rr}
xP_{n}(x;t)=P_{n+1}(x;t)+\al_n(t)P_n(x;t)+\beta_{n}(t)P_{n-1}(x;t),
\ee
with the initial conditions
$$
P_{0}(x;t)=1,\qquad \beta_{0}(t)P_{-1}(x;t)=0.
$$
It can be seen that the recurrence coefficients $\al_n(t)$ and $\bt_n(t)$ have the following integral representations:
$$
\al_n(t)=\frac{1}{h_n(t)}\int_{0}^{\infty}xP_n^2(x;t)w(x;t)dx>0,
$$
\be\label{ir2}
\bt_n(t)=\frac{1}{h_{n-1}(t)}\int_{0}^{\infty}xP_n(x;t)P_{n-1}(x;t)w(x;t)dx.
\ee
Obviously, the expression (\ref{ir2}) is equivalent to
\be\label{be2}
\beta_{n}(t)=\frac{h_{n}(t)}{h_{n-1}(t)}>0.
\ee

Moreover, we have by comparing the coefficients of $x^n$ on both sides of (\ref{rr}) that
\be\label{be1}
\al_{n}(t)=\mathrm{p}(n,t)-\mathrm{p}(n+1,t).
\ee
Taking a telescopic sum of (\ref{be1}), we find
\be\label{sum}
\sum_{j=0}^{n-1}\al_{j}(t)=-\mathrm{p}(n,t).
\ee
As an easy consequence of the three-term recurrence relation (\ref{rr}), we have the Christoffel-Darboux formula
$$
\sum_{j=0}^{n-1}\frac{P_j(x)P_j(y)}{h_j(t)}=\frac{P_n(x)P_{n-1}(y)-P_n(y)P_{n-1}(x)}{h_{n-1}(t)(x-y)},
$$
which plays an important role in the derivation of the ladder operators introduced in the next section.

It is well known that the orthogonal polynomials can be expressed as the determinants \cite[(2.1.6)]{Ismail},
$$
P_n(x;t)=\frac{1}{D_n(t)}\begin{vmatrix}
	\mu_{0}(t)&\mu_{1}(t)&\cdots&\mu_{n}(t)\\
	\mu_{1}(t)&\mu_{2}(t)&\cdots&\mu_{n+1}(t)\\
	\vdots&\vdots&&\vdots\\
	\mu_{n-1}(t)&\mu_{n}(t)&\cdots&\mu_{2n-1}(t)\\
	1&x&\cdots&x^n
\end{vmatrix},
$$
where $D_n(t)$ is the Hankel determinant for the weight (\ref{weight}) defined by
$$
D_{n}(t):=\det(\mu_{i+j}(t))_{i,j=0}^{n-1}=\begin{vmatrix}
	\mu_{0}(t)&\mu_{1}(t)&\cdots&\mu_{n-1}(t)\\
	\mu_{1}(t)&\mu_{2}(t)&\cdots&\mu_{n}(t)\\
	\vdots&\vdots&&\vdots\\
	\mu_{n-1}(t)&\mu_{n}(t)&\cdots&\mu_{2n-2}(t)
\end{vmatrix}
$$
and $\mu_{j}(t)$ is the $j$th moment given by
$$
\mu_{j}(t):=\int_{0}^{\infty}x^{j}w(x;t)dx.
$$
An evaluation of the above integral shows that the moment $\mu_{j}(t)$ can be expressed in terms of the generalized hypergeometric functions.

Furthermore, the Hankel determinant $D_n(t)$ can be expressed as a product of $h_j(t)$ \cite[(2.1.6)]{Ismail},
\be\label{hankel}
D_{n}(t)=\prod_{j=0}^{n-1}h_{j}(t).
\ee
From (\ref{be2}) and (\ref{hankel}), it is easy to see that the recurrence coefficient $\bt_n(t)$ and the Hankel determinant $D_n(t)$ have the following relation:
$$
\bt_n(t)=\frac{D_{n+1}(t)D_{n-1}(t)}{D_n^2(t)}.
$$

The remainder of this paper is organized as follows. In Section 2, we apply the ladder operators and associated compatibility conditions to orthogonal polynomials with the singularly perturbed Airy weight. Based on the identities for the recurrence coefficients and auxiliary quantities, we derive the discrete system satisfied by the recurrence coefficients. We also obtain the second-order differential equation for the orthogonal polynomials. In Section 3, we study the time evolution and find that the recurrence coefficients satisfy the coupled differential-difference equations. The relation between the logarithmic derivative of the Hankel determinant and the recurrence coefficients has also been discussed. In Section 4, we consider the large $n$ asymptotics of the recurrence coefficients by using Dyson's Coulomb fluid approach. Finally, the conclusions are outlined in Section 5.




\section{Ladder operators and the associated compatibility conditions}
Ladder operators for orthogonal polynomials were known to many authors before (one can even go back to Laguerre), but mostly these were obtained case by case. Chen and Ismail \cite{ChenIsmail2} found a general setting for ladder operators which contains all the earlier known cases; see also Ismail \cite[Chapter 3]{Ismail} and Van Assche \cite[Chapter 4]{VanAssche}. The ladder operator approach has been demonstrated to be very useful to analyze the recurrence coefficients of various orthogonal polynomials;
see, e.g., \cite{BCE,ChenDai,ChenIts,CJ,Dai,Filipuk,Min2023,MLC}. The lowering and raising ladder operators for our monic orthogonal polynomials are given by
\be\label{lowering}
\left(\frac{d}{dx}+B_{n}(x)\right)P_{n}(x)=\beta_{n}A_{n}(x)P_{n-1}(x),
\ee
\be\label{raising}
\left(\frac{d}{dx}-B_{n}(x)-\mathrm{v}'(x)\right)P_{n-1}(x)=-A_{n-1}(x)P_{n}(x),
\ee
where $\mathrm{v}(x):=-\ln w(x)$ is the potential and
\be\label{an}
A_{n}(x):=\frac{1}{h_{n}}\int_{0}^{\infty}\frac{\mathrm{v}'(x)-\mathrm{v}'(y)}{x-y}P_{n}^{2}(y)w(y)dy,
\ee
\be\label{bn}
B_{n}(x):=\frac{1}{h_{n-1}}\int_{0}^{\infty}\frac{\mathrm{v}'(x)-\mathrm{v}'(y)}{x-y}P_{n}(y)P_{n-1}(y)w(y)dy.
\ee
Note that we often suppress the $t$-dependence for brevity, and we have $w(0)=w(\infty)=0$.

The functions $A_n(x)$ and $B_n(x)$ are not independent but must satisfy the following compatibility conditions:
\be
B_{n+1}(x)+B_{n}(x)=(x-\al_n) A_{n}(x)-\mathrm{v}'(x), \tag{$S_{1}$}
\ee
\be
1+(x-\al_n)(B_{n+1}(x)-B_{n}(x))=\beta_{n+1}A_{n+1}(x)-\beta_{n}A_{n-1}(x), \tag{$S_{2}$}
\ee
\be
B_{n}^{2}(x)+\mathrm{v}'(x)B_{n}(x)+\sum_{j=0}^{n-1}A_{j}(x)=\beta_{n}A_{n}(x)A_{n-1}(x). \tag{$S_{2}'$}
\ee
The conditions ($S_{1}$) and ($S_{2}$) are essentially a consequence of the three-term recurrence relation (\ref{rr}).
Equation ($S_{2}'$) is obtained from the suitable combination of ($S_{1}$) and ($S_{2}$) and usually gives a better insight into the recurrence coefficients compared with ($S_{2}$) in practice.


For our weight (\ref{weight}), we have
\be\label{vx}
\mathrm{v}(x)=-\ln w(x)=x^3-\lambda\ln x+\frac{t}{x}.
\ee
It follows that
\be\label{vpx}
\mathrm{v}'(x)=3x^2-\frac{\lambda}{x}-\frac{t}{x^2}
\ee
and
\be\label{vp}
\frac{\mathrm{v}'(x)-\mathrm{v}'(y)}{x-y}
=3(x+y)+\frac{\lambda}{xy}+\frac{t}{xy^2}+\frac{t}{x^2y}.
\ee

Substituting (\ref{vp}) into the definition of $A_n(x)$ in (\ref{an}), we find
\begin{align}
A_{n}(x)=&\:\frac{1}{h_{n}}\int_{0}^{\infty}\left[3(x+y)+\frac{\lambda}{xy}+\frac{t}{xy^2}+\frac{t}{x^2y}\right]P_{n}^{2}(y)w(y)dy\nonumber\\[8pt]
=&\:3x+3\al_{n}+\frac{1}{x}\left(\frac{\lambda}{h_{n}}\int_{0}^{\infty}\frac{1}{y}P_{n}^{2}(y)w(y)dy+\frac{t}{h_{n}}\int_{0}^{\infty}\frac{1}{y^2}P_{n}^{2}(y)w(y)dy\right)\no\\[8pt]
&+\frac{t}{x^2h_{n}}\int_{0}^{\infty}\frac{1}{y}P_{n}^{2}(y)w(y)dy.\label{an1}
\end{align}
The formula in the brackets can be simplified through integration by parts. In fact, we have
\begin{align}
\frac{\lambda}{h_{n}}\int_{0}^{\infty}\frac{1}{y}P_{n}^{2}(y)w(y)dy&=\frac{\lambda}{h_{n}}\int_{0}^{\infty}P_{n}^{2}(y)y^{\lambda-1}\mathrm{e}^{-y^3-\frac{t}{y}}dy
=\frac{1}{h_{n}}\int_{0}^{\infty}P_{n}^{2}(y)\mathrm{e}^{-y^3-\frac{t}{y}}dy^\lambda\nonumber\\[8pt]
&=-\frac{1}{h_{n}}\int_{0}^{\infty}P_{n}^{2}(y)w(y)\left(-3y^2+\frac{t}{y^2}\right)dy\nonumber\\[8pt]
&=\frac{3}{h_{n}}\int_{0}^{\infty}y^2P_{n}^{2}(y)w(y)dy-\frac{t}{h_{n}}\int_{0}^{\infty}\frac{1}{y^2}P_{n}^{2}(y)w(y)dy\nonumber\\[8pt]
&=3\left(\al_{n}^2+\beta_{n}+\beta_{n+1}\right)-\frac{t}{h_{n}}\int_{0}^{\infty}\frac{1}{y^2}P_{n}^{2}(y)w(y)dy,\nonumber
\end{align}
where use has been made of the three-term recurrence relation (\ref{rr}) in the last step. It follows that
$$
\frac{\lambda}{h_{n}}\int_{0}^{\infty}\frac{1}{y}P_{n}^{2}(y)w(y)dy+\frac{t}{h_{n}}\int_{0}^{\infty}\frac{1}{y^2}P_{n}^{2}(y)w(y)dy=3\left(\al_{n}^2+\beta_{n}+\beta_{n+1}\right).
$$
Hence, we obtain from (\ref{an1}) that
$$
A_n(x)=3x+3\al_{n}+\frac{3\left(\al_{n}^2+\beta_{n}+\beta_{n+1}\right)}{x}+\frac{t}{x^2h_{n}}\int_{0}^{\infty}\frac{1}{y}P_{n}^{2}(y)w(y)dy.
$$

Similarly, plugging (\ref{vp}) into the definition of $B_n(x)$ in (\ref{bn}) gives
\begin{align}
B_{n}(x)=&\:\frac{1}{h_{n-1}}\int_{0}^{\infty}\left[3(x+y)+\frac{\lambda}{xy}+\frac{t}{xy^2}+\frac{t}{x^2y}\right]P_{n}(y)P_{n-1}(y)w(y)dy\nonumber\\[8pt]
=&\:3\beta_{n}+\frac{1}{x}\left(\frac{\lambda}{h_{n-1}}\int_{0}^{\infty}\frac{1}{y}P_{n}(y)P_{n-1}(y)w(y)dy
+\frac{t}{h_{n-1}}\int_{0}^{\infty}\frac{1}{y^2}P_{n}(y)P_{n-1}(y)w(y)dy\right)\no\\[8pt]
&+\frac{t}{x^2h_{n-1}}\int_{0}^{\infty}\frac{1}{y}P_{n}(y)P_{n-1}(y)w(y)dy.\label{bn1}
\end{align}
Using integration by parts, we find
$$
\frac{\lambda}{h_{n-1}}\int_{0}^{\infty}\frac{1}{y}P_{n}(y)P_{n-1}(y)w(y)dy=3(\al_n+\al_{n-1})\bt_n-n-\frac{t}{h_{n-1}}\int_{0}^{\infty}\frac{1}{y^2}P_{n}(y)P_{n-1}(y)w(y)dy.
$$
That is,
$$
\frac{\lambda}{h_{n-1}}\int_{0}^{\infty}\frac{1}{y}P_{n}(y)P_{n-1}(y)w(y)dy
+\frac{t}{h_{n-1}}\int_{0}^{\infty}\frac{1}{y^2}P_{n}(y)P_{n-1}(y)w(y)dy=3(\al_n+\al_{n-1})\bt_n-n.
$$
Then we obtain the expression of $B_n(x)$ from (\ref{bn1}) that
$$
B_{n}(x)=3\beta_{n}+\frac{3(\al_n+\al_{n-1})\bt_n-n}{x}+\frac{t}{x^2h_{n-1}}\int_{0}^{\infty}\frac{1}{y}P_{n}(y)P_{n-1}(y)w(y)dy.
$$

We summarize the above results in the following lemma.
\begin{lemma}\label{le}
We have
\be\label{anz}
A_n(x)=3x+3\al_{n}+\frac{R_{n}}{x}+\frac{R_n^*}{x^2},
\ee
\be\label{bnz}
B_{n}(x)=3\beta_{n}+\frac{r_n}{x}+\frac{r_n^*}{x^2},
\ee
where $R_{n},\; r_n$ and $R_{n}^*,\; r_{n}^*$ are the auxiliary quantities defined by
\be\label{Rn}
R_{n}:=3\left(\al_{n}^2+\beta_{n}+\beta_{n+1}\right),
\ee
\be\label{rn}
r_n:=3(\al_n+\al_{n-1})\bt_n-n,
\ee
and
\be\label{Rns}
R_{n}^*:=\frac{t}{h_{n}}\int_{0}^{\infty}\frac{1}{y}P_{n}^{2}(y)w(y)dy,
\ee
\be\label{rns}
r_{n}^*:=\frac{t}{h_{n-1}}\int_{0}^{\infty}\frac{1}{y}P_{n}(y)P_{n-1}(y)w(y)dy.
\ee
\end{lemma}

Substituting the expressions of $A_n(x)$ and $B_n(x)$ in (\ref{anz}) and (\ref{bnz}) into ($S_{1}$) and comparing the coefficients of $\frac{1}{x^2}$ and $\frac{1}{x}$ on both sides,  we get
\be\label{re1}
r_{n}^*+r_{n+1}^*=t-\alpha_{n}R_{n}^*
\ee
and
\be\label{re3}
r_n+r_{n+1}=R_{n}^*-\al_{n}R_n+\lambda,
\ee
respectively.

Similarly, substituting (\ref{anz}) and (\ref{bnz}) into ($S_{2}'$) and comparing the coefficients of $\frac{1}{x^4}, \frac{1}{x^3}, \frac{1}{x^2}, \frac{1}{x}$ and $x^0$ on both sides, we obtain
\be\label{re2}
r_{n}^*\left(r_{n}^*-t\right)=\beta_{n}R_{n}^*R_{n-1}^*,
\ee
\be\label{re7}
\left(2r_n-\la\right)r_n^* -t r_n=\bt_n\left(R_n^*R_{n-1}+R_{n-1}^*R_n\right),
\ee
\be\label{re6}
r_n^2-\la r_n-3t\bt_n+6\bt_n r_n^*+\sum_{j=0}^{n-1}R_{j}^*=\bt_n\left(3\al_nR_{n-1}^*+3\al_{n-1}R_n^*+R_nR_{n-1}\right),
\ee
$$
6\bt_nr_n-3\la \bt_n+\sum_{j=0}^{n-1}R_{j}=3\bt_n\left(R_{n}^*+R_{n-1}^*+\al_nR_{n-1}+\al_{n-1}R_n\right),
$$
\be\label{re4}
r_{n}^*+3\bt_n^2+\sum_{j=0}^{n-1}\al_j=\bt_n\left(3\al_n\al_{n-1}+R_{n}+R_{n-1}\right).
\ee
\begin{proposition}\label{pro}
The auxiliary quantities $R_{n}^*$ and $r_{n}^*$ are expressed in terms of the recurrence coefficients as follows:
\be\label{Rne}
R_{n}^*=r_n+r_{n+1}+\al_nR_n-\la,
\ee
\be\label{rne}
r_{n}^*=\frac{\left(r_n+r_{n+1}+\al_nR_n-\la\right)\bt_nR_{n-1}+\left(r_{n-1}+r_{n}+\al_{n-1}R_{n-1}-\la\right)\bt_nR_{n}+tr_n}{2r_n-\la},
\ee
where $R_n$ and $r_n$ are given by (\ref{Rn}) and (\ref{rn}), respectively.
\end{proposition}
\begin{proof}
From (\ref{re3}), we get (\ref{Rne}). Substituting (\ref{Rne}) into (\ref{re7}), we obtain (\ref{rne}).
\end{proof}
With these preparations, we are now ready to derive the discrete system for the recurrence coefficients.
\begin{theorem}\label{the}
The recurrence coefficients $\al_n$ and $\bt_n$ satisfy the following system of nonlinear \textbf{third}-order difference equations:
\begin{align}\label{h1}
&\left[\left(r_n+r_{n+1}+\al_nR_n-\la\right)\bt_nR_{n-1}+\left(r_{n-1}+r_{n}+\al_{n-1}R_{n-1}-\la\right)\bt_nR_{n}+tr_n\right]\no\\
&\times\left[\left(r_n+r_{n+1}+\al_nR_n-\la\right)\bt_nR_{n-1}+\left(r_{n-1}+r_{n}+\al_{n-1}R_{n-1}-\la\right)\bt_nR_{n}-tr_n+\la t\right]\no\\
&=\bt_n(2r_n-\la)^2\left(r_n+r_{n+1}+\al_nR_n-\la\right)\left(r_{n-1}+r_{n}+\al_{n-1}R_{n-1}-\la\right),
\end{align}
\begin{align}\label{h2}
&2\left(r_n+r_{n+1}+\al_nR_n-\la\right)\bt_nR_{n-1}+2\left(r_{n-1}+r_{n}+\al_{n-1}R_{n-1}-\la\right)\bt_nR_{n}+2tr_n\no\\
&+(2r_n-\la)\left[\al_n\left(r_n+r_{n+1}+\al_nR_n-\la-1\right)+3\bt_n^2-3\bt_{n+1}^2-t\right]\no\\
&=(2r_n-\la)\left[\bt_n\left(3\al_n\al_{n-1}+R_{n}+R_{n-1}\right)-\bt_{n+1}\left(3\al_{n+1}\al_{n}+R_{n+1}+R_{n}\right)\right],
\end{align}
where $R_n$ and $r_n$ are given by (\ref{Rn}) and (\ref{rn}), respectively.
\end{theorem}
\begin{proof}

Substituting (\ref{Rne}) and (\ref{rne}) into (\ref{re2}), we obtain (\ref{h1}).
To proceed, replacing $n$ by $n+1$ in (\ref{re4}) and making a difference with (\ref{re4}) give rise to
\be\label{eq1}
r_{n+1}^*-r_{n}^*+3\bt_{n+1}^2-3\bt_n^2+\al_n=\bt_{n+1}\left(3\al_{n+1}\al_{n}+R_{n+1}+R_{n}\right)-\bt_n\left(3\al_n\al_{n-1}+R_{n}+R_{n-1}\right).
\ee
Eliminating $r_{n+1}^*$ from the combination of (\ref{re1}) and (\ref{eq1}), we have
$$
2r_n^*+\al_n(R_n^*-1)+3\bt_n^2-3\bt_{n+1}^2-t=\bt_n\left(3\al_n\al_{n-1}+R_{n}+R_{n-1}\right)-\bt_{n+1}\left(3\al_{n+1}\al_{n}+R_{n+1}+R_{n}\right).
$$
Plugging (\ref{Rne}) and (\ref{rne}) into the above, we arrive at (\ref{h2}).
\end{proof}
\begin{remark}
If one substitutes (\ref{Rne}) and (\ref{rne}) into (\ref{re1}) directly, a \textbf{fourth}-order difference equation for the recurrence coefficients would be obtained.
\end{remark}
\begin{remark}
Using (\ref{sum}), it is seen from (\ref{re4}) that the sub-leading coefficient $\mathrm{p}(n,t)$ can be expressed in terms of the recurrence coefficients $\al_n$ and $\bt_n$.
\end{remark}

At the end of this section, we show that our orthogonal polynomials satisfy a second-order linear ordinary differential equation with the coefficients expressed in terms of $\al_n$ and $\bt_n$.
\begin{theorem}
The monic orthogonal polynomials $P_n(x),\; n=0, 1, 2, \ldots,$ satisfy the following second-order differential equation:
\begin{align}\label{ode}
&P_n''(x)-\bigg(\mathrm{v}'(x)+\frac{A_{n}'(x)}{A_{n}(x)}\bigg)P_n'(x)+\bigg(B_{n}'(x)
-B_{n}^2(x)-\mathrm{v}'(x)B_{n}(x)+\beta_{n}A_{n}(x)A_{n-1}(x)\no\\[8pt]
&-\frac{A_{n}'(x)B_{n}(x)}{A_{n}(x)}\bigg)P_n(x)=0,
\end{align}
where $\mathrm{v}'(x)$ is given by (\ref{vpx}) and
\be\label{anx1}
A_n(x)=3x+3\al_{n}+\frac{R_{n}}{x}+\frac{r_n+r_{n+1}+\al_nR_n-\la}{x^2},
\ee
\be\label{bnx1}
B_{n}(x)=3\beta_{n}+\frac{r_n}{x}+\frac{\left(r_n+r_{n+1}+\al_nR_n-\la\right)\bt_nR_{n-1}+\left(r_{n-1}+r_{n}+\al_{n-1}R_{n-1}-\la\right)\bt_nR_{n}+tr_n}{(2r_n-\la)x^2}
\ee
with $R_n$ and $r_n$ given by (\ref{Rn}) and (\ref{rn}).
\end{theorem}
\begin{proof}
Eliminating $P_{n-1}(x)$ from the ladder operator equations (\ref{lowering}) and (\ref{raising}), we obtain (\ref{ode}). The expressions in (\ref{anx1}) and (\ref{bnx1}) come from Lemma \ref{le} and Proposition \ref{pro}.
\end{proof}

\section{The $t$ evolution and differential-difference equations}
Recall that the recurrence coefficients, the sub-leading coefficient $\mathrm{p}(n,t)$ and the auxiliary quantities all depend on $t$. In this section, we study the evolution of these quantities in $t$. We start from taking a derivative with respect to $t$ in the orthogonality condition
$$
\int_{0}^{\infty}P_{n}(x;t)P_{n-1}(x;t)w(x;t)dx=0.
$$
It follows that
$$
\frac{d}{dt}\mathrm{p}(n,t)=\frac{1}{h_{n-1}(t)}\int_{0}^{\infty}\frac{1}{x}P_{n}(x;t)P_{n-1}(x;t)w(x;t)dx.
$$
From (\ref{rns}) we have
$$
t\frac{d}{dt}\mathrm{p}(n,t)=r_n^{*}.
$$
By making use of (\ref{be1}) and (\ref{eq1}), we find
\begin{align}
t\al_n'(t)&=r_n^{*}-r_{n+1}^{*}\no\\
&=3\bt_{n+1}^2-3\bt_n^2+\al_n+\bt_n\left(3\al_n\al_{n-1}+R_{n}+R_{n-1}\right)-\bt_{n+1}\left(3\al_{n+1}\al_{n}+R_{n+1}+R_{n}\right).\no
\end{align}
Substituting (\ref{Rn}) into the above and simplifying the result, we obtain the differential-difference equation
$$
t\al_n'(t)=\al_n+3\bt_n(\al_n^2+\al_n\al_{n-1}+\al_{n-1}^2+\bt_n+\bt_{n-1})-3\bt_{n+1}(\al_{n+1}^2+\al_{n+1}\al_{n}+\al_{n}^2+\bt_{n+2}+\bt_{n+1}).
$$

On the other hand, differentiating the equality
$$
h_n(t)=	\int_{0}^{\infty}P_{n}^2(x;t)w(x;t)dx
$$
with respect to $t$ gives rise to
$$
h_n'(t)=-\int_{0}^{\infty}\frac{1}{x}P_{n}^2(x;t)w(x;t)dx
$$
Taking account of (\ref{Rns}), we have
\be\label{dp}
t\frac{d}{dt}\ln h_n(t)=-R_n^{*}.
\ee
Using (\ref{be2}) and (\ref{Rne}), it follows that
\begin{align}
t\bt_n'(t)&=\bt_n(R_{n-1}^*-R_n^*)\no\\
&=\bt_n(r_{n-1}-r_{n+1}+\al_{n-1}R_{n-1}-\al_nR_n).\no
\end{align}
Substituting (\ref{Rn}) and (\ref{rn}) into the above produces another differential-difference equation
$$
t\bt_n'(t)=\bt_n\left[2+3\al_{n-2}\bt_{n-1}-3\al_{n+1}\bt_{n+1}+3\al_{n-1}(\al_{n-1}^2+\bt_n+2\bt_{n-1})-3\al_{n}(\al_{n}^2+\bt_n+2\bt_{n+1})\right].
$$
Hence, we obtain the following theorem.
\begin{theorem}
The recurrence coefficients $\al_n$ and $\bt_n$ satisfy the coupled differential-difference equations
$$
t\al_n'(t)=\al_n+3\bt_n(\al_n^2+\al_n\al_{n-1}+\al_{n-1}^2+\bt_n+\bt_{n-1})-3\bt_{n+1}(\al_{n+1}^2+\al_{n+1}\al_{n}+\al_{n}^2+\bt_{n+2}+\bt_{n+1}),
$$
$$
t\bt_n'(t)=\bt_n\left[2+3\al_{n-2}\bt_{n-1}-3\al_{n+1}\bt_{n+1}+3\al_{n-1}(\al_{n-1}^2+\bt_n+2\bt_{n-1})-3\al_{n}(\al_{n}^2+\bt_n+2\bt_{n+1})\right].
$$
\end{theorem}

Finally, we discuss the relation between the logarithmic derivative of the Hankel determinant and the recurrence coefficients.
Let $H_n(t)$ be a quantity related to the logarithmic derivative of the Hankel determinant as follows,
$$
H_n(t):=t\frac{d}{dt}\ln D_n(t).
$$
Using (\ref{hankel}) and (\ref{dp}), we have
$$
H_n(t)=\sum_{j=0}^{n-1}t\frac{d}{dt}\ln h_j(t)=-\sum_{j=0}^{n-1}R_j^*.
$$
In view of (\ref{re6}), we obtain
$$
H_n(t)=r_n^2-\la r_n-3t\bt_n+6\bt_n r_n^*-\bt_n\left(3\al_nR_{n-1}^*+3\al_{n-1}R_n^*+R_nR_{n-1}\right).
$$
Hence, $H_n(t)$ can be expressed in terms of the recurrence coefficients $\al_n$ and $\bt_n$ by using (\ref{Rne}), (\ref{rne}), (\ref{Rn}) and (\ref{rn}). Since the expression is somewhat long, we will not write it down.

\section{Large $n$ asymptotics of the recurrence coefficients}
In this section, we would like to study the asymptotics of the recurrence coefficients $\al_n$ and $\bt_n$ as $n\rightarrow\infty$ by using Dyson's Coulomb fluid approach introduced in the work of Chen and Ismail \cite{ChenIsmail}.

It is well known that Hankel determinants play an important role in random matrix theory (RMT) \cite{Deift,Forrester,Mehta}. This is because Hankel determinants compute the most fundamental objects studied in RMT. For example, the determinants may represent the partition function for a particular random matrix ensemble or they may be related to the
largest and smallest eigenvalue distribution of the ensemble. For our Hankel determinant $D_n(t)$, it can be viewed as the partition function for the unitary ensemble with the singularly perturbed Airy weight
$$
D_n(t)=\frac{1}{n!}\int_{(0,\infty)^n}\prod_{1\leq i<j\leq n}(x_i-x_j)^2\prod_{k=1}^n x_k^\lambda\mathrm{e}^{-x_k^3-\frac{t}{x_k}}dx_k,
$$
where $x_1, x_2, \ldots, x_n$, are the eigenvalues of $n\times n$ Hermitian matrices from the ensemble with the joint probability density function
$$
p(x_1, x_2, \ldots, x_n)=\frac{1}{n!\:D_n(t)}\prod_{1\leq i<j\leq n}(x_i-x_j)^2\prod_{k=1}^n x_k^\lambda\mathrm{e}^{-x_k^3-\frac{t}{x_k}}.
$$

Dyson's Coulomb fluid approach \cite{Dyson} showed that the collection of eigenvalues can be approximated by a continuous fluid with an equilibrium density $\s(x)$ for sufficiently large $n$. It can be seen that our potential $\mathrm{v}(x)$ in (\ref{vx}) satisfies the condition that $x\mathrm{v}'(x)$ increases on $\mathbb{R}^{+}$ when $\la>-1,\;t\geq 0$. In this case, the density $\s(x)$ is supported on a single interval, say $(0,b)$; see \cite[p. 199]{Saff}.

Following \cite{ChenIsmail}, the equilibrium density $\sigma(x)$ is determined by the constrained minimization problem:
$$
\min_{\s}F[\s]\qquad \mathrm{subject}\:\: \mathrm{to}\qquad \int_{0}^{b}\sigma(x)dx=n,
$$
where $F[\s]$ is the free energy defined by
$$
F[\s]:=\int_{0}^{b}\s(x)\mathrm{v}(x)dx-\int_{0}^{b}\int_{0}^{b}\s(x)\ln|x-y|\s(y)dxdy.
$$

It is then found that the density $\s(x)$ satisfies the integral equation
$$
\mathrm{v}(x)-2\int_{0}^{b}\ln|x-y|\s(y)dy=A,\qquad x\in (0,b),
$$
where $A$ is the Lagrange multiplier for the constraint.
Taking a derivative with respect to $x$ for the above equation gives the singular integral equation
\be\label{sie}
\mathrm{v}'(x)-2P\int_{0}^{b}\frac{\sigma(y)}{x-y}dy=0,\qquad x\in (0,b),
\ee
where $P$ denotes the Cauchy principal value. The solution of (\ref{sie}) is given by
$$
\sigma(x)=\frac{1}{2\pi^2}\sqrt{\frac{b-x}{x}}P\int_{0}^{b}\frac{\mathrm{v}'(y)}{y-x}\sqrt{\frac{y}{b-y}}dy.
$$
Substituting (\ref{vpx}) into the above, we obtain
$$
\s(x)=\frac{3  (8 x^2 +4 b x+3 b^2)}{16 \pi }\sqrt{\frac{b-x}{x}}.
$$
It follows that the normalization condition $\int_{0}^{b}\sigma(x)dx=n$ becomes
$$
\frac{15 b^3}{32}=n.
$$

Furthermore, it was shown in \cite{ChenIsmail} that as $n\rightarrow\infty$,
$$
\al_n\sim\frac{b}{2}, \qquad\qquad \bt_n\sim\frac{b^2}{16},
$$
where the symbol $c_n\sim d_n$ means that $\lim\limits_{n\rightarrow\infty}\frac{c_n}{d_n}=1$.
Hence, we obtain the following results.
\begin{theorem}\label{thm}
For fixed parameters $\la>-1,\;t\geq 0$,
the recurrence coefficients of the monic orthogonal polynomials with the weight (\ref{weight}) have the large $n$ asymptotics
$$
\al_n\sim\sqrt[3]{\frac{4n}{15}},\qquad\qquad \bt_n\sim\sqrt[3]{\frac{n^2}{900}}.
$$
\end{theorem}
\begin{remark}
It is interesting to notice that the leading asymptotics of the recurrence coefficients are independent of the parameters $\la$ and $t$.
\end{remark}

\section{Conclusions}
In this paper, we have studied orthogonal polynomials with respect to the singularly perturbed Airy weight. We derived a pair of difference equations and differential-difference equations satisfied by the recurrence coefficients. We proved that the orthogonal polynomials satisfy a linear second-order ordinary differential equation. The relations between the sub-leading coefficient of the monic orthogonal polynomials, the associated Hankel determinant and the recurrence coefficients have also been discussed. Furthermore, we investigated the large $n$ asymptotics of the recurrence coefficients when the parameters $\la$ and $t$ are fixed.

\section*{Acknowledgments}
This work was partially supported by the National Natural Science Foundation of China under grant number 12001212, by the Fundamental Research Funds for the Central Universities under grant number ZQN-902 and by the Scientific Research Funds of Huaqiao University under grant number 17BS402.

\section*{Conflict of Interest}
The authors have no competing interests to declare that are relevant to the content of this article.
\section*{Data Availability Statements}
Data sharing not applicable to this article as no datasets were generated or analysed during the current study.

\end{document}